\documentclass{article}%
\usepackage{graphicx}
\usepackage{amsmath,amssymb,amsfonts}
\usepackage{theorem}
\usepackage{color}
\usepackage{hyperref}
\usepackage[font={small, it}]{caption}
\usepackage[all]{xy}%
\usepackage{amsmath}%
\usepackage{amsfonts}%
\usepackage{amssymb}
\providecommand{\U}[1]{\protect \rule{.1in}{.1in}}
\theoremstyle{change}
\newtheorem{definition}{Definition:}[section]
\newtheorem{proposition}[definition]{Proposition:}
\newtheorem{theorem}[definition]{Theorem:}
\newtheorem{lemma}[definition]{Lemma:}

{\theorembodyfont{\rmfamily}
\newtheorem{remark}[definition]{Remark:}
}
{\theorembodyfont{\rmfamily}

}
\newenvironment{proof}
{{\bf Proof:}}
{\qquad \hspace*{\fill} $\Box$}

\newcommand{\CC}{\mathcal{C}}

\newcommand{\R}{\mathbb{R}}

\begin{document}

	\title{The control set of a linear control system on the two dimensional solvable Lie group}
	\author{V\'{\i}ctor Ayala \thanks{%
			Supported by Proyecto Fondecyt $n^{o}$ 1150292, Conicyt, Chile} \\
		%EndAName
		Universidad de Tarapac\'a\\
		Instituto de Alta Investigaci\'on\\
		Casilla 7D, Arica, Chile\\
		and\\
		Adriano Da Silva \thanks{%
			Supported by Fapesp grant n%
			%TCIMACRO{\U{ba} }%
			%BeginExpansion
			${{}^o}$
			%EndExpansion
			2018/10696-6.}\\
		Instituto de Matem\'atica,\\
		Universidade Estadual de Campinas\\
		Cx. Postal 6065, 13.081-970 Campinas-SP, Brasil.\\
	}
\date{\today }
\maketitle

\begin{abstract}
In this paper we explicitly calculate the control sets associated with a
linear control system on the two dimensional solvable Lie group. We show that a linear control system of such kind admits exactly one control set or infinite control sets depending on some algebraic conditions.

\end{abstract}



\section{Introduction}

The classical linear control systems on Euclidean Spaces
are well known. They are relevant for many theoretical and practical reasons. In particular, they appear in several physical applications (\cite{Leitmann, P-B-G-M, Shell}). They are naturally extended to general Lie groups as showed in  \cite{Markus} for matrices groups, and  then in \cite{AyTi} for any connected Lie group $G$.

In the last twenty-five years, several works addressing controllability,
observability and optimization problems appears for this kind of control
systems, (see \cite{DSAy, DSAyGZ, San, DS, DaJo, Jouan1, Jouan2}).
Furthermore, in \cite{Jouan2} Jouan shows that any affine control system on a
connected manifold that generates a finite dimensional Lie algebra is
diffeomorphic to a linear control system on a Lie group, or on a homogeneous
space. Hence, such kind of generalization is also relevant for the
classification of general affine control systems on abstract connected
manifolds. 

A fundamental notion in control theory is the controllability property of a
control system answering the following question: given an initial state of the
system, is it possible to reach any arbitrary state through admissible
trajectories in positive time? Or better, there are some regions of the space
of state where controllability holds? For instance, in \cite{Shattler} the
authors work out the problem \emph{Optimal controls for a two-compartment
model for cancer chemotherapy \ with quadratic objective. }The space state of
this model is the plane, and its dynamic is given by two matrices which are
elements of the Lie algebra $\mathrm{sl}(2, \mathbb{R})$, of real matrices of
order two and trace zero, see \cite{2-dim} for an algebraic controllability
condition. Therefore, in this case, the controllability property reads as:
given an initial condition $x_{0}$ there exists an admissible control
transferring $x_{0}$ in positive time in a new condition $x_{1}.$ In other
words, is it possible to find a medical strategy to transform an initial level
of disease, at another final level of health$,$ in a positive time$?$. Among
the all possibles estrategies transfering $x_{0}$ into $x_{1}$, you need to
find the optimal control which minimizes the quadratic objective. In our
practical example, to find the optimal control which minimizes the collateral
effects. 

In real life, not any sick condition can be transformed in a health one. Since
in the interior $\mathrm{int}(\mathcal{C)}$ of any control set $\mathcal{C}$ controllability holds, it is fundamental to know about the
existence and uniqueness of control sets, especially those
with non empty interior. And certainly, to characterize the control sets of a
control system in any possible case. The main goal of this paper is to compute
every control set for a linear control system on a solvable Lie group of
dimension two, with and without empty interior.
% In \cite{DSAyGZ} the authors showed, by using a more geometric approach, that linear control systems on solvable Lie groups admits, under strong topological conditions, a unique control set with nonempty interior that contains the identity element.

%From reference \cite{DaJo}, we
%know necessary and sufficient conditions for such systems to be controllable. On the other hand, under the Lie algebra
%rank condition the existence of the control set is assured, \cite{CK}.
%Furthermore, controllability of linear control system was studied in
%\cite{DSAy}, \cite{DS} and {\color{red} dim3} by using a more geometric
%approach by considering the eigenvalues of a derivation associated with the
%drift of the system.

Because of our intention to reach an audience as bigger as possible, we avoid
describing a linear control system as usual through the Lie theory. On the
contrary, we look at linear control systems as special systems evolving on an open half-plane of $\R^2$. Furthermore, aiming a  better understanding and reading of the article, we have include figures of each possible control set. 

The paper is structured as follows: Section 2 contains the basic definition of
control systems, accessible and control sets. We also describe the two
dimensional solvable Lie groups given by the open half-plane $G=\mathbb{R}%
_{+}\times \mathbb{R}$ endowed with its associated non Abelian product. \ In
section 3 we describe case by case the control sets of linear control systems that are conjugated with our initial system. That allow us to know when such sets have empty or nonempty interior and their uniqueness. Finally, at the end of Section 3 we use the group of
automorphisms of $G$ in order to see what are the possibilities for the
control sets of a general linear control system on $G$.

\section{Preliminaries}

\subsection{Control systems and their control sets}

Let $M$ be a $d$-dimensional smooth manifold. A \emph{control system} in $M$
is the family of ordinary differential equations
$$
\dot{x}(t)=f(x(t), u(t)), \; \; \;u\in \mathcal{U},\label{controlsystem}%
$$
where $f: M\times \mathbb{R}^{m}\rightarrow TM$ is a smooth map
and $\mathcal{U}\subset L_{\mathrm{loc}}^{\infty}(\mathbb{R}, \mathbb{R}^{m})$
is the set of the piecewise constant functions whose image are contained in a
compact convex set $\Omega \subset \mathbb{R}^{m}$. For any $x\in M$ and
$u\in \mathcal{U}$ we denote by $\phi(t, x, u)$ the unique solution of
(\ref{controlsystem}) with initial value $x=\varphi(0, x, u)$. The set of
points \emph{reachable from $x$ up to time} $\tau>0$ and the \emph{positive
orbit of} $x$ are given, respectively, by
\[
\mathcal{O}^{+}_{\leq \tau}(x):=\{ \varphi(t, x, u), \; \;t\in[0, \tau
]\; \;u\in \mathcal{U}\} \; \; \; \mbox{ and }\; \; \; \mathcal{O}^{+}(x):=\bigcup
_{t>0}\mathcal{O}^{+}_{t}(x).
\]
With $\mathcal{O}^{-}_{\leq \tau}(x)$ and $\mathcal{O}^{-}(x)$ we denote the
corresponding sets for the time-reversed system. We say that the system
(\ref{controlsystem}) is \emph{locally accessible from} $x$ if $\mathrm{int}%
\mathcal{O}^{\pm}_{\leq \tau}(x)\neq \emptyset$ for all $\tau>0$. A sufficiente
condition for locally accessibility is the Lie algebra rank condition (LARC).
It is satisfied if the Lie algebra $\mathcal{L}$ generated by the vector
fields $x\in M\mapsto f_{u}(x):=f(x, u)$, for $u\in \Omega$, satisfies
$\mathcal{L}(x)=T_{x}M$ for all $x\in M$.

A set $\CC\subset M$ is a \emph{control set} of \ref{controlsystem} if it is
maximal w.r.t. set inclusion with the following properties: 

\begin{enumerate}
\item[(i)] $\CC$ is \emph{controlled invariant}, i.e., for each $x\in \CC$ there
is $u\in \mathcal{U}$ with $\varphi(\mathbb{R}_{+},x,u) \subset \CC$.

\item[(ii)] \emph{Approximate controllability} holds on $\CC$, i.e., $\CC
\subset \operatorname{cl}\mathcal{O}^{+}(x)$ for all $x\in \CC$.
\end{enumerate}

Following \cite{CK}, Proposition 3.2.4., any subset $\CC$ of $M$ with nonempty
interior that is maximal with property (ii) in the above definition is a
control set.

Let us consider $\psi:M\rightarrow N$ to be a diffeomorphism and consider
\[
\dot{x}(t)=f(x(t), u(t))\; \; \mbox{ and }\; \; \dot{y}(t)=g(y(t), u(t))
\; \; \;u\in \mathcal{U}%
\]
control systems on $M$ and $N$, respectively. We say that $\psi$
\textit{conjugates} the control systems if
\[
g(\psi(x), u)=(d\psi)_{x}f(x, u) \; \; \mbox{ for any } \; \;x\in G,
u\in \mathcal{U}.
\]
In this cases we say that the control systems are equivalent. The conjugation
of control systems will be used ahead several times.

\subsection{Two-dimensional linear control systems}

In this section we analyze linear control systems on the two-dimensional
solvable Lie group.

Let us denote by $G=\mathbb{R}_{+}\times \mathbb{R}$ the open half-plane of
$\mathbb{R}^{2}$ and endow it with the product
\[
(x_{1}, y_{1})\cdot(x_{2}, y_{2})=(x_{1}x_{2}, y_{2}+x_{2}y_{1}).
\]
It is a standard fact that the $G$ is in fact a Lie group and, up to an
isomorphism, is the unique two-dimensional solvable Lie group.

Following \cite{DaJo}, a \textit{linear} vector field on $G$ is a vector field
of the form
\[
\mathcal{X}(x, y)=(0, a(x-1)+by), \; \; \mbox{ for some }\; \;(a, b)\in
\mathbb{R}^{2}.
\]
Moreover, a simple calculation shows that the left-invariant vector fields of
$G$ are of the form
\[
Y(x, y)=(x\alpha, x\beta), \; \; \mbox{ for some }\; \;(\alpha, \beta
)\in \mathbb{R}^{2}.
\]
Let $\Omega=[u_{*}, u^{*}]$ with $u_{*}<0<u^{*}$. A \textit{linear control
system} on $G$ is a system of the form
\[
\dot{(x, y)}=\mathcal{X}(x, y)+uY(x, y), \; \; \; \mbox{ with }\; \;u\in \Omega,
\]
where $\mathcal{X}$ and $Y$ are nontrivial vector fields. In coordinates,
\[
\hspace{-2cm}(\Sigma)\hspace{2cm}\left \{
\begin{array}
[c]{l}%
\dot{x}=u\alpha x\\
\dot{y}=a(x-1)+by+ux\beta
\end{array}
\right. , \; \; \mbox{ where }\; \;u\in \Omega \; \; \mbox{ and }\; \;(a, b), (\alpha,
\beta)\in \mathbb{R}^{2}\setminus \{(0, 0)\}.
\]
A simple calculation shows that
\[
\mathcal{L}(x, y)=\mathrm{span}\{(u\alpha x, a(x-1)+by+ux\beta), (0,
ux(a\alpha+b\beta)), \; \;u\in \Omega \}
\]
and $\mathcal{L}(x, y)=\mathbb{R}^{2}$ for all $(x, y)\in G$ if and only if
$\alpha(a\alpha+b\beta)\neq0$, that is, the LARC holds for $\Sigma$ if and
only if $\alpha(a\alpha+b\beta)\neq0$.

In order to analyze the control sets of $\Sigma$ it will be necessary to
conjugate the system in order to simplify it. Because of that we need the
following notion: An automorphism of $G$ is a map $\psi:G\rightarrow G$ that
preserves the product, that is,
\[
\psi((x_{1},y_{1})\cdot(x_{2},y_{2}))=\psi(x_{1},y_{1})\cdot \psi(x_{2}%
,y_{2}),\; \;(x_{1},y_{1}),(x_{2},y_{2})\in G.
\]
The automorphisms $\psi:G\rightarrow G$ have the form
$$
\psi(x,y)=(x,c(x-1)+dy),\; \;d\in \mathbb{R}^{\ast}.\label{auto}%
$$
Moreover, the automorphisms of $G$ preserves linear and left-invariant vector
fields and hence conjugates linear control systems. This fact will be used
ahead several times in order to simplify calculations.

\section{The control sets of linear control systems on $G$}

The aim of this section, is analyze the control sets of a given linear control
system. In order to do that we conjugate the given system by an automorphism
in to simplify the calculations and make the problem more abordable to deal with.

The next result, which we will prove through the following sections summarize
our findings.

\begin{theorem}
\label{main}  For the control system $\Sigma$ it holds that 

\begin{itemize}

\item[1.] If $\alpha=0$ then $\Sigma$ has infinite control sets; 

\item[2.] If $\alpha \neq0$ then $\Sigma$ admits a unique control set that has
nonempty interior if and only if the LARC holds. 
\end{itemize}
\end{theorem}

The proof of the theorem is divided in the next sections.

\subsection{The case $\alpha=a\alpha+b\beta=0$}

Since $(a,b),(\alpha,\beta)\in \mathbb{R}^{2}\setminus \{(0,0)\}$ the above
condition implies that $\alpha=b=0$ and $a,\beta \in \mathbb{R}^{\ast}$ and
therefore, the system $\Sigma$ is of the form,
\[
\left \{
\begin{array}
[c]{l}%
\dot{x}=0\\
\dot{y}=a(x-1)+ux\beta
\end{array}
\right.  ,\; \; \mbox{ where }\; \;u\in \Omega,
\]
whose solutions starting at $(x,y)\in G$ are given by
\[
\varphi(t,(x,y),u)=(x,(a(x-1)+ux\beta)t+y),\; \;t\in \mathbb{R}.
\]
We claim
\[
\mathcal{C}_{x}:=\{x\} \times \mathbb{R}%
\; \mbox{ is a control set for any }\;x\in(1-\varepsilon,1+\varepsilon).
\]

In fact, for any $x\in \mathbb{R}_{+}$ the line $\{x\} \times \mathbb{R}\subset
G$ is invariant by the solutions of $\Sigma$. On the other hand, if
$x\in(1-\varepsilon, 1+\varepsilon)$ and there exist $u_{1}, u_{2}\in \Omega$
such that $a(x-1)+u_{1}x\beta<0$ and $a(x-1)+u_{2}x\beta>0$. Hence, it turns
out that
\[
\varphi_{2}(t, (x, y), u_{1})\rightarrow-\infty \; \; \mbox{ and } \; \; \varphi
_{2}(t, (x, y), u_{2})\rightarrow+\infty \; \mbox{ as }\;t\rightarrow+\infty
\]
where $\varphi_{2}$ stands for the second component of $\varphi$. Therefore,
$\mathcal{O}^{+}(x, y)=\{x\} \times \mathbb{R}$ for any $y\in \mathbb{R}$
implying that $\mathcal{C}_{x}$ is a control set for any $x\in(1-\varepsilon,
1+\varepsilon)$ (see Figure \ref{fig1}). In particular, the control system $\Sigma$ admits an infinite number of
control sets.

\begin{figure}[h]
\begin{center}
\includegraphics[scale=1]{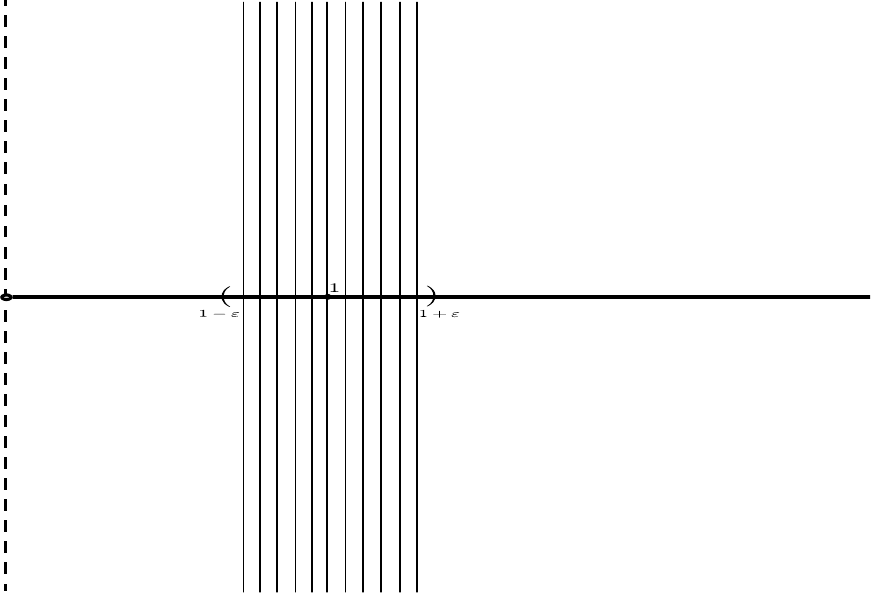}
\end{center}
\par
\caption{The control sets of $\Sigma$.}
\label{fig1} \end{figure}

\subsection{The case $\alpha=0$ and $a\alpha+b\beta \neq0$}

In this case, we necessarily have that $b,\beta \in \mathbb{R}^{\ast}$.
Therefore, the map $\psi(x,y):=(x,a(x-1)+by)$ is an automorphism of $G$ that
conjugates $\Sigma$ and the linear control system,
\begin{equation}
\left \{
\begin{array}
[c]{l}%
\dot{x}=0\\
\dot{y}=by+uxb\beta
\end{array}
\right.  ,\; \; \mbox{ where }\; \;u\in \Omega,\label{2}%
\end{equation}
whose solutions starting at $(x,y)\in G$ are given by
\[
\varphi(t,(x,y),u)=(x,\mathrm{e}^{tb}y+(\mathrm{e}^{tb}-1)ux\beta
),\; \;t\in \mathbb{R}.
\]
Let us analyze the case where $b<0$ since the other case is analogous. For any
given $x\in \mathbb{R}_{+}$ we use the compactness of $\Omega$ to define
\[
y_{1}(x):=\min \{-ux\beta,\; \;u\in \Omega \} \; \; \mbox{ and }\; \;y_{2}%
(x):=\max \{-ux\beta,\; \;u\in \Omega \}.
\]
Since $0\in \operatorname{int}\Omega$ we get $y_{1}(x)<0<y_{2}(x)$. We claim
that the set
\[
\mathcal{C}_{x}=\{x\} \times \left[  y_{1}(x),y_{2}(x)\right]
\]
is a positively-invariant control set of (\ref{2}). In fact, for any
$u\in \Omega$ and $y\in \left[  y_{1}(x),y_{2}(x)\right]  $ it holds that
\[
y_{1}(x)-\varphi_{2}(t,(x,y),u)=y_{1}(x)-\mathrm{e}^{tb}y+(1-\mathrm{e}%
^{tb})ux\beta
\]%
\[
\leq y_{1}(x)-\mathrm{e}^{tb}y_{1}(x)+(1-\mathrm{e}^{tb})ux\beta
\leq(1-\mathrm{e}^{tb})(y_{1}(x)+ux\beta)\leq0\implies y_{1}(x)\leq \varphi
_{2}(t,(x,y),u)
\]
and
\[
y_{2}(x)-\varphi_{2}(t,(x,y),u)=y_{2}(x)-\mathrm{e}^{bt}y+(1-\mathrm{e}%
^{tb})ux\beta
\]%
\[
\geq y_{2}(x)-\mathrm{e}^{bt}y_{2}(x)+(1-\mathrm{e}^{tb})ux\beta
=(1-\mathrm{e}^{bt})(y_{2}(x)+ux\beta)\geq0\implies \varphi_{2}(t,(x,y),u)\leq
y_{2}(x)
\]
showing that $\mathcal{C}_{x}$ is positively-invariant.

Let $u_{1},u_{2}\in \Omega$ such that $y_{i}(x)+u_{i}x\beta=0,\;i=1,2$. Then,
for any $y\in(y_{1}(x),y_{2}(x))$ we have that
\[
\varphi_{2}(t,(x,y),u_{i})\rightarrow-u_{i}x\beta=y_{i}(x)
\]
implying that $\mathcal{C}_{x}\subset \operatorname{cl}\mathcal{O}^{+}(x,y)$
for any $y\in(y_{1}(x),y_{2}(x))$. By continuity and invariance, we get that
\[
\mathcal{C}_{x}=\operatorname{cl}\mathcal{O}^{+}%
(x,y),\; \; \mbox{ for any }\; \;(x,y)\in \mathcal{C}_{x}%
\]
concluding the proof.

A simple calculation shows that $[y_{1}(x_{0}), y_{2}(x_{0})]\subset[y_{1}(x_{1}), y_{2}(x_{1})] $ if $x_0<x_1$. Hence (\ref{2}) admits an infinite number of control sets (see Figure \ref{fig2}).

\begin{figure}[h]
\begin{center}
\includegraphics[scale=1]{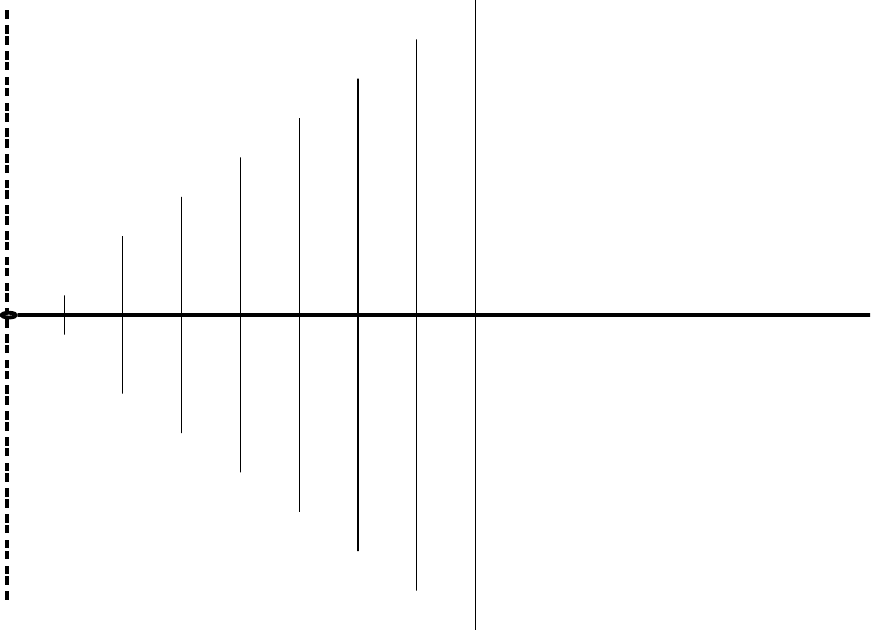}
\end{center}
\par
\caption{The control sets of (\ref{2}).}
\label{fig2} \end{figure}

\subsection{The case $\alpha \neq0$ and $a\alpha+b\beta=0$}

In this case, we necessarily have that $b\neq0$. Let us consider the
automorphis of $G$ given by $\psi(x, y):=(x, y-(x-1)\beta \alpha^{-1})$. We
have that $\psi$ conjugates $\Sigma$ and the linear control system
\begin{equation}
\label{4}\left \{
\begin{array}
[c]{l}%
\dot{x}=u\alpha x\\
\dot{y}=by
\end{array}
\right. , \; \; \mbox{ where }\; \;u\in \Omega,
\end{equation}
whose solutions starting at $(x, y)\in G$ are given by concatenations of
flows
\[
\varphi(t, (x, y), u)=(\mathrm{e}^{tu\alpha}x, \mathrm{e}^{tb}y),
\; \;t\in \mathbb{R}.
\]
For the above control system, the only control set is given by $\mathcal{C}%
=\mathbb{R}_{+}\times \{0\}$.

Let us first show that $\mathcal{C}$ is a control set. We notice that
$\varphi(t, \mathcal{C}, u)\subset \mathcal{C}$ for any $u\in \Omega$ and
$t\in \mathbb{R}$. On the other hand, if $\varphi_{1}$ is the first component
of $\varphi$, it holds that
\[
\left \{
\begin{array}
[c]{l}%
\varphi_{1}(t, (x, 0), u)\rightarrow+\infty, \; \; \mbox{ for }\; \;t\rightarrow
+\infty \; \; \mbox{ when }\; \; \alpha u>0\\
\varphi_{1}(t, (x, 0), u)\rightarrow0, \; \; \; \; \mbox{ for }\; \;t\rightarrow
+\infty \; \; \mbox{ when }\; \; \alpha u<0
\end{array}
\right. ,
\]
implying that $\mathcal{C}=\operatorname{cl}(\mathcal{O}^{+}(x, 0))$ for any
$x\in \mathbb{R}_{+}$ and consequently that $\mathcal{C}$ is a control set.

Let us assume that $b<0$ and show the uniqueness. The case $b>0$ is analogous.
In order to do it is enough to show that no point in $G$ outside $\mathcal{C}$
satisfies condition (ii) in the definition of control sets.

In fact, let $(x_{0}, y_{0})\in G$ with $y_{0}\neq0$. By the form of the
solutions $\operatorname{cl}(\mathcal{O}^{+}(x_{0}, y_{0}))$ is comprehended
between the lines $y=0$ and $y=y_{0}$. Moreover, if $(x_{1}, y_{1}%
)\in \mathcal{O}^{+}(x_{0}, y_{0})$ there exists $t_{0}>0$ such that
$y_{1}=\mathrm{e}^{bt_{0}}y_{0}<y_{0}$ if $y_{0}>0$ and $y_{1}=\mathrm{e}%
^{bt_{0}}y_{0}>y_{0}$ if $y_{0}<0$. Thus $(x_{0}, y_{0})$ is outside the
region determine by the lines $y=0$ and $y=y_{1}$ and therefore $(x_{0},
y_{0})\notin \operatorname{cl}(\mathcal{O}^{+}(x_{1}, y_{1}))$ if $(x_{1},
y_{1})\in \operatorname{cl}(\mathcal{O}^{+}(x_{0}, y_{0}))$. Consequently, no
points in $\mathcal{M}\setminus \mathcal{C}$ can be inside a control set
implying the uniqueness of $\mathcal{C}$ (see Figure \ref{fig3}).

\begin{figure}[h]
\begin{center}
\includegraphics[scale=1]{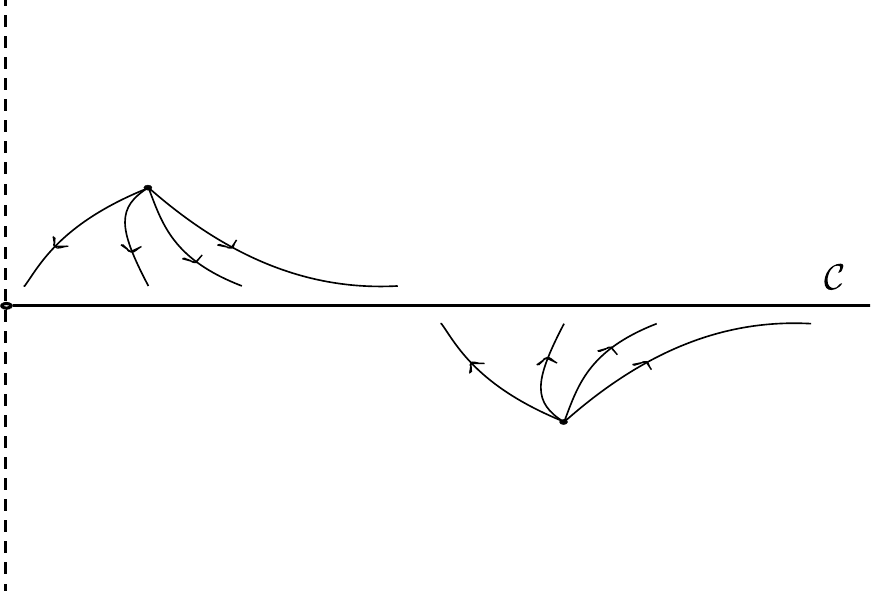}
\end{center}
\par
\caption{The control set $\CC$ of (\ref{4}) and the behavior of the solutions outside $\CC$.}
\label{fig3} \end{figure}

\subsection{The case $\alpha(a\alpha+b\beta)\neq0$}

In this section, we analyze the control sets under the LARC. As Theorem
\ref{main} states in this case we have the associated control set is unique
and has nonempty interior. We will divide the analysis in the following two sections.

\subsubsection{The case $b=0$}

In this situation, we necessarily have that $a\neq0$ and so $\psi(x, y):=(x,
a^{-1}y-\beta \alpha^{-1}(x-1))$ is a diffeomorphism that conjugates $\Sigma$
and the control system
\begin{equation}
\label{3}\left \{
\begin{array}
[c]{l}%
\dot{x}=u\alpha x\\
\dot{y}=x-1
\end{array}
\right. , \; \; \mbox{ where }\; \;u\in \Omega,
\end{equation}
whose solutions starting at $(x, y)\in G$ are given by concatenations of the
flows
\[
\varphi(t, (x, y), u)=\left( \mathrm{e}^{u\alpha t}x, \frac{(\mathrm{e}%
^{u\alpha t}-1)x}{u\alpha}-t+y\right) , \; \;t\in \mathbb{R}, \; \;u\neq0
\]
and
\[
\varphi(t, (x, y), 0)=\left( x, (x-1)t+y\right) , \; \;t\in \mathbb{R}, \; \;u=0.
\]

Before showing that the control system (\ref{3}) is controllable it is important to notice that in \cite{DaJo} the authors prove that $b=0$ and
the LARC are equivalent to the controllability of a linear control system. The
difference here is that we show explicitly ``the way´´such controllability is obtained. This is certainly worth since one can use it
in optimability problems concerning such systems.

Let then $(x_{1}, y_{1}), (x_{2}, y_{2})\in G$ and assume that $x_{1}<1<x_{2}%
$. It holds:

\begin{itemize}
\item[(i)] $(x_{2}, y_{2})\in \mathcal{O}^{+}(x_{1}, y_{1}):$

In fact, let $u\in \Omega$ with $u\alpha>0$. Since $\varphi_{1}(t, (x_{1},
y_{1}), u)=\mathrm{e}^{u\alpha t}x_{1}$ there exists $t_{1}>0$ such that
$\varphi_{1}(t_{1}, (x_{1}, y_{1}), u)=x_{2}$. By considering $y_{2}^{\prime
}:=\varphi_{2}(t_{1}, (x_{1}, y_{1}), u)$ we have that:

\begin{itemize}

\item[1.] Assume $y_{2}^{\prime}\leq y_{2}$. Since $x_{2}>1$ it follows that
$t_{2}:=\frac{y_{2}-y_{2}^{\prime}}{x_{2}-1}\geq0$. Consequently
\[
\varphi(t_{2}, (x_{2}, y_{2}^{\prime}), 0)=(x_{2}, (x_{2}-1)t_{2}%
+y_{2}^{\prime})=(x_{2}, y_{2})\; \; \implies \varphi(t_{2}, \varphi(t_{1},
(x_{1}, y_{1}), u), 0)=(x_{2}, y_{2}).
\]
Therefore $(x_{2}, y_{2})\in \mathcal{O}^{+}(x_{1}, y_{1})$; 

\item[2.] Assume $y_{2}^{\prime}>y_{2}$. Since $y\mapsto \frac{(\mathrm{e}%
^{u\alpha t_{1}}-1)x_{1}}{u\alpha}-t_{1}+y$ is strictly increasing, there
exists $y_{1}^{\prime}<y_{1}$ such that
\[
\varphi_{2}(t_{1}, (x_{1}, y_{1}^{\prime}), u)=\frac{(\mathrm{e}^{u\alpha
t_{1}}-1)x_{0}}{u\alpha}-t_{1}+y_{1}^{\prime}=y_{2}.
\]
Since $x_{1}<1$ we have that $t_{0}=\frac{y_{1}^{\prime}-y_{1}}{x_{1}-1}>0$
and hence
\[
\varphi(t_{0}, (x_{1}, y_{1}), 0)=(x_{1}, (x_{1}-1)t_{0}+y_{1})=(x_{1},
y_{1}^{\prime})\; \; \implies \; \; \varphi(t_{1}, \varphi(t_{0}, (x_{1}, y_{1}),
u), 0)=(x_{2}, y_{2})
\]
implying that $(x_{2}, y_{2})\in \mathcal{O}^{+}(x_{1}, y_{1}) $ (see Figure \ref{fig4}). 
\end{itemize}

\bigskip

\item[(ii)] $(x_{1}, y_{1})\in \mathcal{O}^{+}(x_{2}, y_{2}):$

Let $u\in \Omega$ with $u\alpha<0$. Since $\varphi_{1}(t, (x_{2}, y_{2}),
u)=\mathrm{e}^{u\alpha t}x$ there exists $t_{1}>0$ such that $\varphi
_{1}(t_{1}, (x_{2}, y_{2}), u)=x_{1}$. By considering $y_{1}^{\prime}%
:=\varphi_{2}(t_{1}, (x_{2}, y_{2}), u)$ we have that:

\begin{itemize}

\item[1.] Assume $y_{1}^{\prime}\geq y_{1}$. Since $x_{1}<1$ it follows that
$t_{2}:=\frac{y_{1}-y_{1}^{\prime}}{x_{1}-1}\geq0$. Consequently
\[
\varphi(t_{2}, (x_{1}, y_{1}^{\prime}), 0)=(x_{1}, (x_{1}-1)t_{2}%
+y_{1}^{\prime})=(x_{1} y_{1})\; \; \implies \varphi(t_{2}, \varphi(t_{1},
(x_{2}, y_{2}), u), 0)=(x_{1}, y_{1}).
\]
Therefore $(x_{1}, y_{1})\in \mathcal{O}^{+}(x_{2}, y_{2})$; 

\item[2.] Assume $y_{1}^{\prime}<y_{1}$. Since $y\mapsto \frac{(\mathrm{e}%
^{u\alpha t_{1}}-1)x_{2}}{u\alpha}-t_{1}+y$ is strictly increasing, there
exists $y_{2}^{\prime}>y_{2}$ such that
\[
\varphi_{2}(t_{0}, (x_{2}, y_{2}^{\prime}), u)=\frac{(\mathrm{e}^{u\alpha
t}-1)x_{2}}{u\alpha}-t+y_{2}^{\prime}=y_{1}.
\]
Since $x_{2}>1$ we have that $t_{0}=\frac{y_{2}^{\prime}-y_{2}}{x_{2}-1}>0$
and hence
\[
\varphi(t_{0}, (x_{2}, y_{2}), 0)=(x_{2}, (x_{2}-1)t_{0}+y_{2})=(x_{2},
y_{2}^{\prime})\; \; \implies \; \; \varphi(t_{1}, \varphi(t_{0}, (x_{2}, y_{2}),
0), u)=(x_{1}, y_{1})
\]
implying that $(x_{1}, y_{1})\in \mathcal{O}^{+}(x_{2}, y_{2}).$ 
\end{itemize}
\end{itemize}

\bigskip

\begin{figure}[h]
	\begin{center}
		\includegraphics[scale=1.2	]{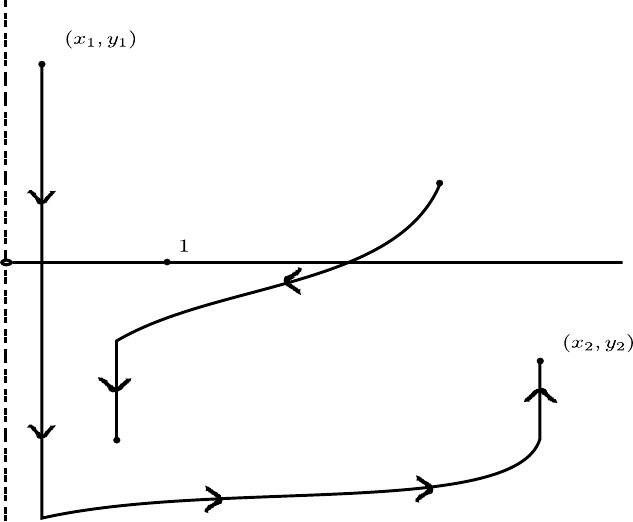}
	\end{center}
	\par
	\caption{Solutions of (\ref{4}) connecting distinct points.}
	\label{fig4} \end{figure}

Now we are able to prove a controllability result.

\begin{theorem}
If $b=0$ the only control set of $\Sigma$ is the whole space $G$. 
\end{theorem}

\begin{proof}
By conjugation it is enough to show that $G$ is the control set of the control
system (\ref{3}). Let us consider $(x_{1}, y_{1}), (x_{2}, y_{2})\in
G\setminus \{(1, y), y\in \mathbb{R}\}$. If $x_{1}>1$ and $x_{2}>1$ we can
consider $u\in \Omega$ with $u\alpha<0$ and, since $\mathrm{e}^{u\alpha t}%
x_{1}\rightarrow0$ as $t\rightarrow+\infty$ there exists $t_{0}>0$ such that
$x_{1}^{\prime}=\mathrm{e}^{u\alpha t_{0}}x_{1}<1$. By the above, it holds
that $(x_{2}, y_{2})\in \mathcal{O}^{+}(\varphi(t_{0}, (x_{1}, y_{1}), u))$ and
hence $(x_{2}, y_{2})\in \mathcal{O}^{+}(x_{1}, y_{1})$. An analogous analysis
for the case $x_{1}<1$ and $x_{2}<1$ gives us also $(x_{2}, y_{2}%
)\in \mathcal{O}^{+}(x_{1}, y_{1})$ and consequently
\[
G\setminus \{(1, y), y\in \mathbb{R}\} \subset \mathcal{O}^{+}(x, y),
\; \; \mbox{ for any }\; \;(x, y)\in G\setminus \{(1, y), y\in \mathbb{R}\}.
\]
Since $G\setminus \{(1, y), y\in \mathbb{R}\}$ is certainly dense in $G$ we get
that $G\subset \operatorname{cl}(\mathcal{O}^{+}(x, y))$ for any $(x, y)\in
G\setminus \{(1, y), y\in \mathbb{R}\}$. On the other hand, for any $(x, y)\in
G$ there exists $u\in \Omega$ and $t>0$ such that $\varphi(t, (x, y), u)\in
G\setminus \{(1, y), y\in \mathbb{R}\}$. Finally,
\[
G\subset \operatorname{cl}(\mathcal{O}^{+}(\varphi(t, (x, y), u)))\subset
\operatorname{cl}(\mathcal{O}^{+}(x, y))\subset G
\]
implying that $G$ is the only control set of (\ref{3}) and concluding the
proof. 
\end{proof}

\subsubsection{The case $b\neq0$}

Since the sign of $b$ in not relevant for the proof, we only consider the case
$b<0$. By considering the diffeomorphism of $G$ defined by $\psi(x, y):=(x,
\gamma^{-1}(a(x-1)+by))$, where $\gamma=a\alpha+b\beta \neq0$, it follows that
$\Sigma$ is conjugated to the control system
\begin{equation}
\label{system1}\left \{
\begin{array}
[c]{l}%
\dot{x}=u\alpha x\\
\dot{y}=by+ux
\end{array}
\right. , \; \; \mbox{ where }\; \;u\in \Omega.
\end{equation}
The solutions starting at $(x, y)\in G$ of (\ref{system1}) are given by
concatenations of the flows
\begin{equation}
\label{constant}%
\begin{array}
[c]{c}%
\varphi(t, (x, y), u)=(\mathrm{e}^{u\alpha t}x, m_{u}(\mathrm{e}^{u\alpha
t}-\mathrm{e}^{bt})x+\mathrm{e}^{bt}y), \; \; \mbox{ for }u\alpha \neq
b\; \; \mbox{ and }\; \;m_{u}=\frac{u}{u\alpha-b}\\
\hspace{2cm}\\
\mbox{ and }\; \; \; \varphi(t, (x, y), b\alpha^{-1})=\left( \mathrm{e}^{bt}x,
\mathrm{e}^{tb}(y+tb\alpha^{-1} x)\right) , \; \;t\in \mathbb{R},
\; \; \mbox{ when }\;u\alpha=b.
\end{array}
\end{equation}
For any $u\in \Omega$ with $u\alpha \neq b$ we denote by $r_{u}$ the \emph{ray}
of $G$ given by
\[
r_{u}:=\{(x, y)\in G;\; \;y-m_{u}x=0\},
\]
that is, $r_{u}$ is the intersection with $G$ of the line by the origin of
$\mathbb{R}^{2}$ with inclination $m_{u}$.

Define $m:\mathbb{R}\setminus \{b\alpha^{-1}\} \rightarrow \mathbb{R}$ to be the
map given by $u\mapsto m_{u}$. It is straightforward to see that (see Figure \ref{fig6})

\begin{itemize}

\item[1.] $m_{u}=\frac{-b}{(u\alpha-b)^{2}}>0$ and so, $m$ is strictly
increasing on $(-\infty, b\alpha^{-1})$ and on $(b\alpha^{-1}, +\infty)$; 

\item[2.] $\lim_{u\rightarrow \pm \infty}m_{u}=\alpha^{-1}$ and $\lim
_{u\rightarrow(b\alpha^{-1})^{\pm}}m_{u}=\mp \infty$. 
\end{itemize}

\begin{figure}[h]
\begin{center}
\includegraphics[scale=1.3]{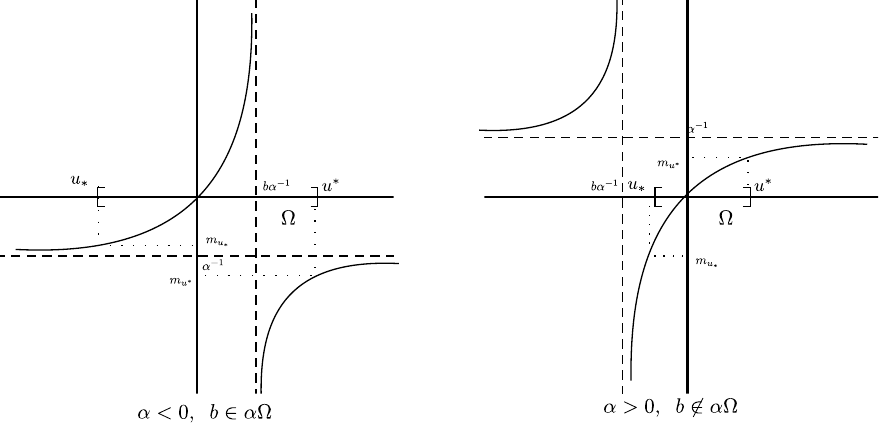}
\end{center}
\par
\caption{Behavior of $m_u$.}
\label{fig6} \end{figure}

Let us consider $B:=\{u\in \Omega; \;u\alpha-b>0\}$. Since we are assuming
$b<0$ and $0\in \mathrm{int}\Omega$ we necessarily have that $u\alpha-b>0$ for
some $u\in \Omega$ and consequently $B\neq \emptyset$. For the solutions of the
above control system we have the following:

\begin{proposition}
\label{solutions}  For any $t>0$ and $u\in \Omega$ it holds: 

\begin{itemize}

\item[1.] $\varphi_{t, u}(r_{u})\subset r_{u}$; 

\item[2.] $\varphi(t, (x, y_{1}+y_{2}), u)=\varphi(t, (x, y_{1}), u)+(0,
\mathrm{e}^{bt}y_{2})$; 

\item[3.] For $u\alpha \neq0$ we obtain
\[
m_{u_{*}}\varphi_{1}(t, (x, y), u)-\varphi_{2}(t, (x, y), u)\leq
\mathrm{e}^{bt}(m_{u_{*}}x-y)\; \; \mbox{ if }\; \;u_{*}\in B;
\]
\[
\varphi_{1}(t, (x, y), u)-m_{u^{*}}\varphi_{1}(t, (x, y), u)\leq
\mathrm{e}^{bt}(y-m_{u^{*}}x)\; \; \mbox{ if }\; \;u^{*}\in B.
\]

\item[4.] For $u=b\alpha^{-1}$ we get
\[
m_{u_{*}}\varphi_{1}(t, (x, y), u)-\varphi_{2}(t, (x, y), u)\leq
\mathrm{e}^{bt}(m_{u_{*}}x-y)\; \; \mbox{ if }\; \; \alpha<0,\; \; \mbox{ and }\; \;
\]
\[
\varphi_{1}(t, (x, y), u)-m_{u^{*}}\varphi_{1}(t, (x, y), u)\leq
\mathrm{e}^{bt}(y-m_{u^{*}}x)\; \; \mbox{ if }\; \; \alpha>0.
\]

\end{itemize}
\end{proposition}

\begin{proof}
Since a general solution of (\ref{system1}) is given by concatenations of the
flows in (\ref{constant}) it is enough to show the proposition for $u\in
\Omega$.

1. In fact, if $(x, y)\in r_{u}$ we have that $y=m_{u}x$ and for any $(x,
y)\in r_{u}$
\[
\varphi(t, (x, y), u)=(\mathrm{e}^{\alpha ut}x, m_{u}(\mathrm{e}^{u\alpha
t}-\mathrm{e}^{bt})x+\mathrm{e}^{bt}m_{u}x)=(\mathrm{e}^{\alpha ut}x,
\mathrm{e}^{\alpha ut}m_{u}x)=\mathrm{e}^{\alpha ut}(x, y).
\]

2. In fact, if $u\alpha \neq b$ we get
\[
\varphi(t, (x, y_{1}+y_{2}), u)=(\mathrm{e}^{u\alpha t}x, m_{u}(\mathrm{e}%
^{u\alpha t}-\mathrm{e}^{bt})x+\mathrm{e}^{bt}(y_{1}+y_{2}))
\]
\[
=(\mathrm{e}^{u\alpha t}x, m_{u}(\mathrm{e}^{u\alpha t}-\mathrm{e}%
^{bt})x+\mathrm{e}^{bt}y_{1})+(0, \mathrm{e}^{bt}y_{2})=\varphi(t, (x, y_{1}),
u)+(0, \mathrm{e}^{bt}y_{2})
\]
and for $u\alpha=b$
\[
\varphi(t, (x, y_{1}+y_{2}), b\alpha^{-1})=\left( \mathrm{e}^{bt}x,
\mathrm{e}^{tb}(y_{1}+y_{2}+tb\alpha^{-1} x)\right)
\]
\[
\left( \mathrm{e}^{bt}x, \mathrm{e}^{tb}(y_{1}+tb\alpha^{-1} x)\right) +(0,
\mathrm{e}^{tb}y_{2})=\varphi(t, (x, y_{1}), b\alpha^{-1})+(0, \mathrm{e}%
^{bt}y_{2}).
\]

3. Let us analyze the case for $u_{*}\in B$. We have that
\[
m_{u_{*}}\varphi_{1}(t, (x, y), u)-\varphi_{2}(t, (x, y), u)=m_{u_{*}%
}\mathrm{e}^{u\alpha t}x-m_{u}\mathrm{e}^{u\alpha t}x+m_{u}\mathrm{e}%
^{bt}x-\mathrm{e}^{bt}y=(m_{u_{*}}-m_{u})\mathrm{e}^{u\alpha t}x+\mathrm{e}%
^{bt}x-\mathrm{e}^{bt}y
\]

However, if $u\alpha-b>0$ then $m_{u_{*}}\leq m_{u}$ implying that $(m_{u_{*}%
}-m_{u})\mathrm{e}^{u\alpha t}\leq(m_{u_{*}}-m_{u})\mathrm{e}^{bt}$. If
$u\alpha-b<0$ then $m_{u_{*}}>m_{u}$ and consequently $(m_{u_{*}}%
-m_{u})\mathrm{e}^{u\alpha t}\leq(m_{u_{*}}-m_{u})\mathrm{e}^{bt}$.
Therefore,
\[
m_{u_{*}}\varphi_{1}(t, (x, y), u)-\varphi_{2}(t, (x, y), u)\leq(m_{u_{*}%
}-m_{u})\mathrm{e}^{bt}+m_{u}\mathrm{e}^{bt}x-\mathrm{e}^{bt}y=\mathrm{e}%
^{bt}(m_{u_{*}}x-y), \; \;u\alpha-b\neq0.
\]

4. If $\alpha<0$ we obtain $b\alpha^{-1}>0$ and hence
\[
m_{u_{*}}\varphi_{1}(t, (x, y), b\alpha^{-1})-\varphi_{2}(t, (x, y),
b\alpha^{-1})=m_{u_{*}}\mathrm{e}^{bt}x-\mathrm{e}^{bt}(y+tb\alpha^{-1}%
x)\leq \mathrm{e}^{bt}(m_{u_{*}}x-y).
\]

\end{proof}

\begin{remark}
By concatenations, itens 2 to 4 are also valid for any piecewise constant
function $u\in \mathcal{U}$. 
\end{remark}

Let us consider the set
\[
\mathcal{C}:=\bigcup_{u\in B}r_{u}.
\]
Our aim is to show that $\mathcal{C}$ is in fact the only control set of
(\ref{system1}). In order to that the next lemma, concerning the main
properties of $\mathcal{C}$, will be central.

\begin{lemma}
\label{lemma}  Let $u\in \Omega$. It holds: 

\begin{itemize}

\item[1.] If $b\notin \alpha \Omega$, then
\[
\mathcal{C}=\{(x, y)\in G; \; \; m_{u_{*}}x\leq y \leq m_{u^{*}}x \};
\]

\item[2.] If $b\in \alpha \Omega$, then
\[
\mathcal{C}=\{(x, y)\in G; \; \;y\leq m_{u^{*}}x\} \; \; \mbox{ if }\; \; \alpha
>0\; \; \mbox{ and }
\]
\[
\mathcal{C}=\{(x, y)\in G; \; \;y\geq m_{u_{*}}x\} \; \; \mbox{ if }\; \; \alpha<0
\]

\item[3.] For any $u_{1}, u_{2}\in \mathrm{int}B$ with $u_{1}\neq u_{2}$ there
exists $t_{0} >0$ and $u\in \Omega$ such that
\[
\varphi(t_{0}, r_{u_{1}}, u)=r_{u_{2}};
\]

\item[4.] The subset $\mathcal{C}$ is positively-invariant; 
\end{itemize}
\end{lemma}

\begin{proof} 1. Since we are assuming that $b\notin \alpha \Omega$ it holds that
$\Omega \subset(-\infty, b\alpha^{-1})$ or $\Omega \subset(b\alpha^{-1},
+\infty)$. Being that $\Omega=[u_{*}, u^{*}]$ and $m$ is strictly crescent we
get
\[
m_{u_{*}}\leq m_{u}\leq m_{u^{*}} \; \; \mbox{ for all }\; \;u\in \Omega.
\]
Therefore, if $(x, y)\in \mathcal{C}$, it turns out $y=m_{u}x$ for some
$u\in \Omega$ implying that $m_{u_{*}}x\leq y\leq m_{u^{*}}x$. On the other
hand, if $(x, y)\in G$ with $m_{u_{*}}x\leq y\leq m_{u^{*}}x$ then
$y/x\in[m_{u_{*}}, m_{u^{*}}]$ which by continuity implies the existence of
$u\in \Omega$ such that $m_{u}=y/x$ implying that $(x, y)\in r_{u}$ and
concluding the proof.

2. Let us show the case $\alpha<0$ since the other case is analogous. A simple
calculation shows that
\[
B=\Omega \cap(-\infty, b\alpha^{-1})=[u_{*}, b\alpha^{-1}),
\; \; \mbox{ with }\; \;b\alpha^{-1}>0
\]
and consequently
\[
m_{u_{*}}=\inf_{u\in B}m_{u}<0\; \; \mbox{ and }\; \; \sup_{u\in B}m_{u}=+\infty.
\]

If $(x, y)\in \mathcal{C}$ there exists $u\in B$ such that $y=m_{u}x\geq
m_{u^{*}}x$ showing that $\mathcal{C}\subset \{(x, y)\in G; \; \;y\geq m_{u_{*}%
}x\}$. On the other hand, for any $(x, y)$ such that $y/x\geq m_{u^{*}}$ we
have that $y/x\in[m_{u_{*}}, +\infty)=m(B)$. So, there exists $u\in B$ with
$y/x=m_{u}$ implying that $(x, y)\in \mathcal{C}$ and as stated
\[
\mathcal{C}=\{(x, y)\in G; \; \;y\geq m_{u_{*}}x\}.
\]

3. Since $\varphi(t, \lambda(x, y), u)=\lambda \varphi(t, (x, y), u)$ for any
$(x, y)\in G$, $\lambda \in \mathbb{R}^{*}_{+}$ and $u\in \Omega$, it is enough
to show that $\varphi(t_{0}, r_{u_{1}}, u)\cap r_{u_{2}}\neq \emptyset$ for
some $t_{0}>0$ and $u\in \Omega$.

Let then $(x, y)\in r_{u_{1}}$, $u\in B$ and consider the continuous map
$g_{u}:\mathbb{R}\rightarrow \mathbb{R}$ given by
\[
t\in \mathbb{R}\mapsto g_{u}(t):=\frac{\varphi_{2}(t, (x, y), u)}{\varphi
_{1}(t, (x, y), u)}.
\]
A simple calculation shows that
\[
g_{u}(t)=m_{u}+\mathrm{e}^{-t(u\alpha-b)}\left( m_{u_{1}}-m_{u}\right) ,
\; \; \mbox{ since }\; \;m_{u_{1}}=\frac{y}{x}.
\]
Consequently,
\begin{equation}
\label{1}g(0)=m_{u_{1}} \; \; \; \mbox{ and }\; \; \;g(t)\rightarrow m_{u}
\; \; \mbox{ as }\; \;t\rightarrow+\infty
\end{equation}
because $u\alpha-b>0$. Thus, if $u_{1}, u_{2}\in \mathrm{int}B$ there exists
$u\in B$ such that $m_{u_{2}}\in(m_{u_{1}}, m_{u})$ or $m_{u_{2}}\in(m_{u},
m_{u_{1}})$, depending if $m_{u_{1}}$ is greater or smaller than $m_{u_{2}}$.
By the continuity of $g_{u}$ and (\ref{1}) there exists $t_{0}>0$ such that
$g_{u}(t_{0})=m_{u_{2}}$ and
\[
\varphi_{2}(t_{0}, (x, y), u)=m_{u_{2}}\varphi_{1}(t_{0}, (x, y),
u)\implies \varphi(t_{0}, (x, y), u)\in r_{u_{2}}%
\]
which concludes the proof.

4. Let $(x, y)\in \mathrm{int}\mathcal{C}$. By the proof of items 1. and 2.
above, there exists $u\in \mathrm{int}B$ such that $(x, y)\in r_{u}$. Moreover,
by the proof of item 3. $\varphi(t, (x, y), v)\in \mathrm{int}\mathcal{C}$ for
any $t>0$ and $v\in B$.

Therefore, it is enough to assume that $b\in \alpha \Omega$ and show that
$\varphi(t, (x, y), u)\in \mathcal{C}$ for any $u\in \Omega \setminus B$.

Let us analyze the case where $\alpha<0$. In this situation, $B=[u_{*},
b\alpha^{-1})$ and consequently, we only have to show that $\varphi(t, (x, y),
u)\in \mathcal{C}$ for any $u\in[b\alpha^{-1}, u^{*}]$.

However, by the properties of $m$ we know that $m_{u^{*}}>\alpha^{-1}>m_{u}$
for any $u\in(b\alpha^{-1}, u^{*}]$. According to the definition of $g_{u}$ in
item 3. we obtain $g_{u}(0)=y/x\geq m_{u_{*}}$ since $(x, y)\in \mathrm{int}%
\mathcal{C}$ and
\[
g^{\prime}_{u}(t)=-(u\alpha-b)\mathrm{e}^{-t(u\alpha-b)}(y/x-m_{u})>0,
\; \; \mbox{ if }\; \;u\in(b/\alpha, u^{*}].
\]
Therefore, $g([0, +\infty)\subset[m_{u_{*}}, +\infty)$ showing that
$\varphi(t, (x, y), u)\in \mathcal{C}$ for any $t>0$ and any $u\in(b\alpha
^{-1}, u^{*}]$. On the other hand, if $u\alpha=b$ we have that
\[
\frac{\varphi_{2}(t, (x, y), b\alpha^{-1})}{\varphi_{1}(t, (x, y),
b\alpha^{-1})}=\frac{\mathrm{e}^{bt}(y+t b\alpha^{-1}x)}{\mathrm{e}^{bt}%
x}=y/x+t b\alpha^{-1}>y/x\geq m_{u^{*}}%
\]
implying that $\varphi(t, (x, y), b\alpha^{-1})\in \mathcal{C}$ and
consequently that
\[
\varphi(t, (x, y), u)\in \mathcal{C}, \; \; \mbox{ for any }\; \;t>0
\; \mbox{ and }\;u\in \Omega.
\]
Since $\mathrm{int}\mathcal{C}$ is dense in $\mathcal{C}$ we get that
$\varphi(t, \mathcal{C}, u)\subset \mathcal{C}$ for any $t>0$ and $u\in \Omega$.
Since the solutions of the control system are given by concatenations of the
above flows, we get that $\mathcal{C}$ is positively-invariant as stated. 
\end{proof}

\begin{remark}
	Let us notice that itens 1. and 2. of Lemma \ref{lemma} shows that $\CC$ is a cone in $G$ with (open) wedge on $(0, 0)\in\R^2$ (see Figure \ref{fig5} below).
\end{remark}

\begin{figure}[h]
\begin{center}
\includegraphics[scale=1]{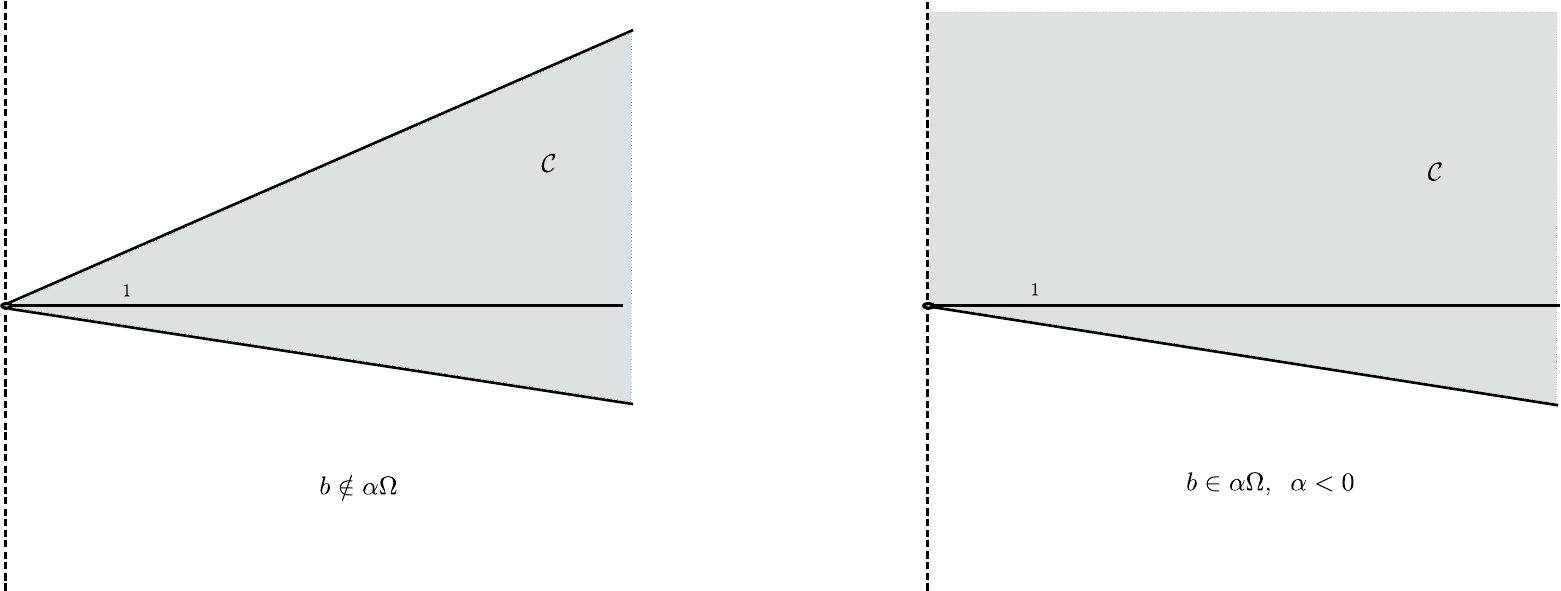}
\end{center}
\par
\caption{The possibilities for the control set of (\ref{system1}).}
\label{fig5} \end{figure}

\bigskip

We are now able to prove the main result of this section.

\begin{theorem}
If $b<0$ the unique control set of (\ref{system1}) is $\mathcal{C}$. 
\end{theorem}

\begin{proof}
We will show that $\mathcal{C}=\operatorname{cl}(\mathcal{O}^{+}(x, y))$ for
any $(x, y)\in \mathcal{C}$.

Take $(x_{1}, y_{1}), (x_{2}, y_{2})\in \mathrm{int}\mathcal{C}$ and consider
$u_{1}, u_{2}\in \mathrm{int}B$, $u_{1}\neq u_{2}$, such that $m_{u_{i}}%
=y_{i}/x_{i}$, $i=1, 2$. Consider also $u\in \Omega$ and $t_{0}>0$ such that
$\varphi(t_{0}, r_{u_{1}}, u)=r_{2}$ and denote $(x_{2}^{\prime},
y_{2}^{\prime}):=\varphi(t_{0}, (x_{1}, y_{1}), u)$.

Assume $m_{u_{1}}m_{u_{2}}<0$. Since this condition is equivalent to
$u_{1}u_{2}<0$ we have:
\begin{itemize}

\item[(i)] $u_{1}\alpha<0<u_{2}\alpha$ 

\begin{itemize}

\item[1.] $|(x_{2}^{\prime}, y_{2}^{\prime})|\leq|(x_{2}, y_{2})|$: Since
$\varphi(t, (x_{2}, y_{2}), u_{2})=\mathrm{e}^{u_{2}\alpha t}(x_{2}, y_{2})$
there exists $t_{1}\geq0$ such that
\[
\varphi(t_{1}, (x_{2}, y_{2}), u_{2})=(x_{2}^{\prime}, y_{2}^{\prime
})\; \; \implies \; \; \varphi(t_{1}, \varphi(t_{0}, (x_{1}, y_{1}), u),
u_{2})=(x_{2}, y_{2}).
\]
Hence $(x_{2}, y_{2})\in \mathcal{O}^{+}(x_{1}, y_{1})$; 

\item[2.] $|(x_{2}^{\prime}, y_{2}^{\prime})|>|(x_{2}, y_{2})|$: Since
$\varphi(t, (x_{1}, y_{1}), u_{1})=\mathrm{e}^{u_{1}\alpha t}(x_{1}, y_{1})$
and $\lambda:=|(x_{2}, y_{2})|/|(x_{2}^{\prime}, y_{2}^{\prime})|<1$, there
exists $t_{1}>0$ such that $\varphi(t_{1}, (x_{1}, y_{1}), u_{1}%
)=\lambda(x_{1}, y_{1})$. Therefore,
\[
\varphi(t_{0}, \varphi(t_{1}, (x_{1}, y_{1}), u_{1}), u)=\lambda \varphi(t_{0},
(x_{1}, y_{1}), u)=\lambda(x_{2}^{\prime}, y_{2}^{\prime})\in r_{u_{2}}.
\]
However,
\[
|\varphi(t_{0}, \varphi(t_{1}, (x_{1}, y_{1}), u_{1}), u)|=\lambda
|(x_{2}^{\prime}, y_{2}^{\prime})|=|(x_{2}, y_{2})|\implies \varphi(t_{0},
\varphi(t_{1}, (x_{1}, y_{1}), u_{1}), u)=(x_{2}, y_{2})
\]
and hence $(x_{2}, y_{2})\in \mathcal{O}^{+}(x_{1}, y_{1})$. 
\end{itemize}

\bigskip

\item[(ii)] $u_{2}\alpha<0<u_{1}\alpha$ 

\begin{itemize}

\item[1.] $|(x_{2}^{\prime}, y_{2}^{\prime})|\geq|(x_{2}, y_{2})|$: Since
$\varphi(t, (x_{2}, y_{2}), u_{2})=\mathrm{e}^{u_{2}\alpha t}(x_{2}, y_{2})$
there exists $t_{1}\geq0$ such that
\[
\varphi(t_{1}, (x_{2}, y_{2}), u_{2})=(x_{2}^{\prime}, y_{2}^{\prime
})\; \; \implies \; \; \varphi(t_{1}, \varphi(t_{0}, (x_{1}, y_{1}), u),
u_{2})=(x_{2}, y_{2})
\]
and hence $(x_{2}, y_{2})\in \mathcal{O}^{+}(x_{1}, y_{1})$; 

\item[2.] $|(x_{2}^{\prime}, y_{2}^{\prime})|<|(x_{2}, y_{2})|$: Since
$\varphi(t, (x_{1}, y_{1}), u_{1})=\mathrm{e}^{u_{1}\alpha t}(x_{1}, y_{1})$
and $\lambda:=|(x_{2}, y_{2})|/|(x_{2}^{\prime}, y_{2}^{\prime})|>1$, there
exists $t_{1}>0$ such that $\varphi(t_{1}, (x_{1}, y_{1}), u_{1}%
)=\lambda(x_{1}, y_{1})$. Therefore,
\[
\varphi(t_{0}, \varphi(t_{1}, (x_{1}, y_{1}), u_{1}), u)=\lambda \varphi(t_{0},
(x_{1}, y_{1}), u)=\lambda(x_{2}^{\prime}, y_{2}^{\prime})\in r_{u_{2}}.
\]
However,
\[
|\varphi(t_{0}, \varphi(t_{1}, (x_{1}, y_{1}), u_{1}), u)|=\lambda
|(x_{2}^{\prime}, y_{2}^{\prime})|=|(x_{2}, y_{2})|\implies \varphi(t_{0},
\varphi(t_{1}, (x_{1}, y_{1}), u_{1}), u)=(x_{2}, y_{2})
\]
and hence $(x_{2}, y_{2})\in \mathcal{O}^{+}(x_{1}, y_{1})$. 
\end{itemize}
\end{itemize}

Let us assume now that $m_{u_{1}}m_{u_{2}}>0$. Again, this condition is
equivalent to $u_{1}u_{2}>0$. Let us analyze the case where $u_{1}%
\alpha<0\; \; \mbox{ and }\; \;u_{2}\alpha<0$, since the other possibility is analogous.

In this case, by considering $u\in \mathrm{int}B$ with $u\alpha>0$ we have by
the proof of the Lemma \ref{lemma} that
\[
\frac{\varphi_{2}(t, (x_{1}, y_{1}), u)}{\varphi_{1}(t, (x, y), u)}\rightarrow
m_{u}, \; \;t\rightarrow+\infty \; \; \implies \; \; \frac{\varphi_{2}(t, (x, y),
b\alpha^{-1})}{\varphi_{1}(t, (x, y), u)}>0, \; \; \mbox{ for some }t>0.
\]
By the above, we have that
\[
(x_{2}, y_{2})\in \mathcal{O}^{+}(\varphi(t, (x_{1}, y_{1}),
u))\; \; \; \mbox{ and hence }\; \; \;(x_{2}, y_{2})\in \mathcal{O}^{+}(x_{1},
y_{1})
\]

By the arbitrariness of the choosen points, we obtain $\mathrm{int}%
\mathcal{C}\subset \mathcal{O}^{+}(x, y)$ for any $(x, y)\in \mathrm{int}%
\mathcal{C}$. Since $\mathcal{C}$ is closed, we get $\mathcal{C}%
\subset \operatorname{cl}(\mathcal{O}^{+}(x, y))$ for all $(x, y)\in
\mathrm{int}\mathcal{C}$. On the other hand, if $(x, y)\in \mathcal{C}$ and
$u\in \mathrm{int}B$, by the proof of Lemma \ref{lemma} $\varphi(t, (x, y),
u)\in \mathrm{int}\mathcal{C}$ and consequently
\[
\mathcal{C}\subset \operatorname{cl}(\mathcal{O}^{+}(\varphi(t, (x, y),
u)))\subset \operatorname{cl}(\mathcal{O}^{+}(x, y))\subset \mathcal{C}%
\; \; \implies \; \; \mathcal{C}=\operatorname{cl}(\mathcal{O}^{+}(x, y)), \; \;(x,
y)\in \mathcal{C}%
\]
showing that $\mathcal{C}$ is a control set.

Now we prove the uniqueness of $\mathcal{C}$. If
\[
\mathcal{C}_{1}:=\{(x, y)\in G; \;y<m_{u_{*}}%
x\} \; \; \; \mbox{ and }\; \; \; \mathcal{C}_{2}:=\{(x, y)\in G; \;y>m_{u_{*}}x\}
\]
we have that
\[
G\setminus \mathcal{C}=\left \{
\begin{array}
[c]{cc}%
\mathcal{C}_{1}\cup \mathcal{C}_{2} & \mbox{ if }b\notin \alpha \Omega,\\
\mathcal{C}_{1} & \mbox{ if }\; \;b\in \alpha \Omega \; \; \mbox{ and }\; \; \alpha
<0\\
\mathcal{C}_{2} & \mbox{ if }\; \;b\in \alpha \Omega \; \; \mbox{ and }\; \; \alpha>0
\end{array}
\right.
\]
Let $(x, y)\in \mathcal{C}_{1}$ and assume that $\varphi(t, (x, y),
u)\in \mathcal{C}_{1}$ for some $u\in \mathcal{U}$ and $t>0$. Then, if $(x_{t},
y_{t})=(\varphi_{1}(t, (x, y), u), \varphi_{2}(t, (x, y), u)$, Proposition \ref{solutions} gives that  
\[
\varphi(s, \varphi(t, (x, y), u), u^{\prime})=\varphi(s, (x_{t}, m_{u_{*}%
}x_{t}), u^{\prime})+(0, \mathrm{e}^{bs}(y_{t}-m_{u_{*}}x_{t}))
\]
implying that
\[
|\varphi(s, \varphi(t, (x, y), u), u^{\prime})-\varphi(s, (x_{t}, m_{u_{*}%
}x_{t}), u^{\prime})|=\mathrm{e}^{bs}\left| \varphi_{2}(t, (x, y),
u)-m_{u_{*}}\varphi_{1}(t, (x, y), u)\right|
\]
\[
=\mathrm{e}^{bs}\bigl(m_{u_{*}}\varphi_{1}(t, (x, y), u)-\varphi_{2}(t, (x,
y), u)\bigr)\leq \mathrm{e}^{b(s+t)}m_{u_{*}}x-y.
\]
However, $(x_{t}, m_{u_{*}}x_{t})\in \mathcal{C}$ and $\mathcal{C}$ is
positively-invariant, hence
\[
\varphi(s, \varphi(t, (x, y), u), u^{\prime})\in N_{\epsilon(t)}%
(\mathcal{C}), \; \; \mbox{ where }\; \; \epsilon(t)=\mathrm{e}^{bt}m_{u_{*}%
}x-y>0.
\]
Since $s>0$ and $u^{\prime}\in \mathcal{U}$ where arbitrary, we have that (see Figure \ref{fig7})
\begin{equation}
\label{bla}\mathcal{O}^{+}(\varphi(t, (x, y), u))\subset N_{\varepsilon
(t)}(\mathcal{C}).
\end{equation}
Consequently, if $\widetilde{\mathcal{C}}$ is a control set and there exists
$(x, y)\in \widetilde{\mathcal{C}}\cap \mathcal{C}_{1}\neq \emptyset$, by the
controlled invariance of $\widetilde{\mathcal{C}}$ there exists $u\in
\mathcal{U}$ such that $\varphi(t, (x, y), u)\in \widetilde{\mathcal{C}}$ for
any $t>0$. If there is $t_{0}>0$ such that $\varphi(t_{0}, (x, y),
u)\in \mathcal{C}$ then $\widetilde{\mathcal{C}}\subset \operatorname{cl}%
\mathcal{O}^{+}(\varphi(t_{0}, (x, y), u))=\mathcal{C}$.  On the other hand,
if $\varphi(t, (x, y), u)\notin \mathcal{C}$ for any $t>0$ then $\varphi(t, (x,
y), u)\in \mathcal{C}_{1}$ for any $t>0$. By equation \ref{bla} we get
\[
\widetilde{\mathcal{C}}\subset \bigcap_{t>0}\mathcal{O}^{+}(\varphi(t, (x, y),
u))\subset \bigcap_{t>0}\operatorname{cl}(N_{\epsilon(t)}(\mathcal{C}%
))=\mathcal{C},
\]
where we used that $\mathcal{C}$ is closed and that $\epsilon
(t)\rightarrow0$ as $t\rightarrow+\infty$. In any case we must have
$\widetilde{\mathcal{C}}\subset \mathcal{C}$ implying $\mathcal{C}_{1}%
\cap \mathcal{C}\neq \emptyset$ which is a contradiction since $\mathcal{C}%
_{1}\subset G\setminus \mathcal{C}$. Therefore, there is no control set
intersecting $\mathcal{C}_{1}$.

In an analogous way we show that there is no control set intersecting
$\mathcal{C}_{2}$ and therefore $\mathcal{C}$ is the only control set,
concluding the proof. 
\end{proof}

\begin{figure}[h]
\begin{center}
\includegraphics[scale=1.3]{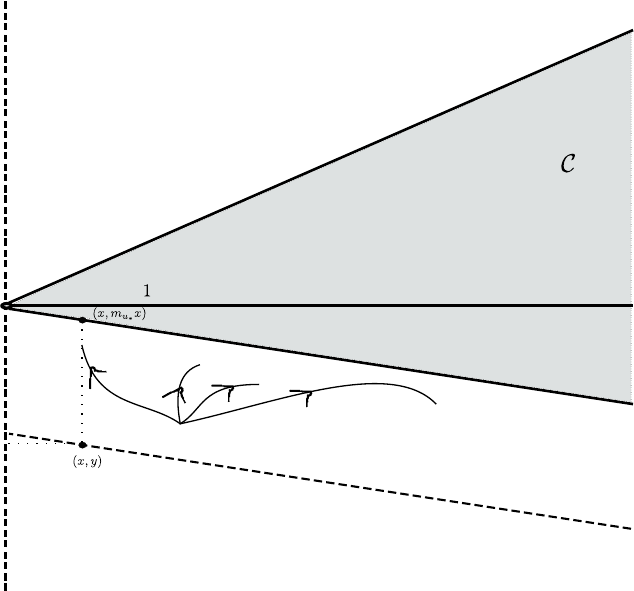}
\end{center}
\par
\caption{An $\epsilon(t)$-neighborhood of $\CC$ and the behavior of the solutions outside $\CC$.}
\label{fig7} \end{figure}

\begin{remark}
Let us notice that the previous result shows that $\Sigma$ admit exactly one
control set and it has nonempty interior. For more general Lie groups, the
authors showed in \cite{DSAyGZ} that linear control systems admits, under
strong topological conditions, exactly one control set with nonempty interior.
However, there are no information about the control sets with empty interior.
\end{remark}

\subsection{Automorphisms of $G$ and control sets}

By the calculations in the previous sections, any arbitrary linear control system on $G$ is conjugated to one of the linear control systems (\ref{2}), (\ref{4}), (ref{3}) or (\ref{system1}) by an automorphism. Therefore, the control sets of $\Sigma$ can be reobtained from the control sets of the above system by considering the preimage of the automorphism in question.

With that in mind we have the following geometric view of the control sets of $\Sigma$.

\begin{theorem}
	For the linear control system $\Sigma$ it holds:
	\begin{itemize}
		\item[1.] $\alpha=a\alpha+b\beta=0$ and any vertical line close to $(1, 0)$ is a control set;
		\item[2.] $\alpha=0$ and $a\alpha+b\beta\neq 0$, and the control sets are vertical segments intersecting
		$$\{(x, y)\in G; \;y=-ab^{-1}(x-1)\};$$
		\item[3.] $\alpha\neq 0$ and $a\alpha+b\beta=0$, and $\Sigma$ admits only the control set
		$$\{(x, y)\in G; \; y=\beta\alpha^{-1}(x-1)\};$$
		\item[4.] $\alpha(a\alpha+b\beta)\neq 0$ with $b=0$ and the unique control set is the whole $G$;
		\item[5.] $\alpha(a\alpha+b\beta)\neq 0$ with $b\neq 0$ and the unique control set is a cone in $G$ with (open) edge on the point $(0, ab^{-1})$.
	\end{itemize}
\end{theorem}

The proof of the previous result is straightforward and follows directly from the following facts concerning an arbitrary automorphism. Let $\psi(x, y)=(x, c(x-1)+dy)$ be an automorphim of $G$. It holds: 

\begin{itemize}
	\item[(i)] If $r\subset G$ is a ray, $\psi(r)=l\cap G$ where $l$ is a line passing by $(0, -c)$;
	
	\item[(ii)] $\psi$ preserves any vertical line in $G$. 
\end{itemize}


\begin{thebibliography}{99}                                                                                               %


\bibitem {DSAy} \newblock V. Ayala and A. Da Silva,
\newblock \emph{Controllability of Linear Control Systems on Lie Groups with
Semisimple Finite Center},  \newblock SIAM Journal on Control and Optimization
55 No 2 (2017), 1332-1343.

\bibitem {DSAyGZ} \newblock V. Ayala, A. Da Silva and G. Zsigmond,
\newblock \emph{Control sets of linear systems on Lie groups.}
\newblock Nonlinear Differential Equations and Applications - NoDEA 24 No 8
(2017), 1 - 15.

\bibitem {San} \newblock V. Ayala and L.A.B. San Martin,
\newblock \emph{Controllability properties of a class of control systems on
Lie groups},  \newblock  Lecture Notes in Control and Information Sciences 258
(2001), 83 - 92.

\bibitem {AyTi} \newblock V. Ayala and J. Tirao, \newblock \emph{Linear control
systems on Lie groups and Controllability}, \newblock Eds. G. Ferreyra et al.,
Amer. Math. Soc., Providence, RI, 1999.

\bibitem {2-dim}V. Ayala and L. San Martin. \emph{Controllability of }%
$2$-\emph{dimensional Bilinear Systems: Restricted Controls, Discrete-Time}.
Proy, Vol. 18, pp. 207-223, 1999.

%	\bibitem{BiGS}
%	\newblock R.M. Bianchini and G. Stefani,
%	\newblock \emph{Normal Local Controllability of Order One},
%	\newblock Int J Control 39 No 4 (1984), 701-714.


\bibitem {CK} \newblock F. Colonius and W. Kliemann,  \newblock \emph{The
Dynamics of Control},  \newblock Birkh\"{a}user, Boston, 2000.

\bibitem {DS} \newblock A. Da Silva,  \newblock \emph{Controllability of
linear systems on solvable Lie groups},  \newblock SIAM Journal on Control and
Optimization 54 No 1 (2016), 372-390.

\bibitem {DaJo} \newblock M. Dath and P. Jouan,
\newblock \emph{Controllability of Linear Systems on Low Dimensional Nilpotent
and Solvable Lie Groups},  \newblock Journal of Dynamics and Control Systems
22 N0 2 (2016), 207-225.

%\bibitem {Elliot} \newblock D. L. Elliott,  \newblock \emph{Bilinear Control
%Systems: Matrices in Action},  \newblock Springer 2009.

%	\bibitem{HH1}
%	\newblock H. Hermes,
%	\newblock \emph{Lie algebras of vector fields and local approximations of attainable sets},
%	\newblock Journal of Dynamics and Control Systems, 16 No 5 (1977), 715-727.


%	\bibitem{HH2}
%	\newblock H. Hermes,
%	\newblock \emph{Control systems which generate decomposable Lie Algebras},
%	\newblock J. Diff. Eqs. 44 (1982), 166-187.


%\bibitem{Hur}
%\newblock A. Hurwitz,
%\newblock \emph{\"{U}ber algebraische Gebilde mit Eindeutigen Transformationen in sich},
%\newblock Mathematische Annalen 41 No 3 (1983), 403-442.


\bibitem {Jouan1} \newblock Ph. Jouan,  \newblock \emph{Controllability of
linear systems on Lie group},  \newblock J. Dyn. Control Syst. 17 (2011), 591-616.

\bibitem {Jouan2} \newblock Ph. Jouan,  \newblock \emph{Equivalence of Control
Systems with Linear Systems on Lie Groups and Homogeneous Spaces},
\newblock ESAIM: Control Optimization and Calculus of Variations, 16 (2010) 956-973.

%\bibitem {Jurd} \newblock V. Jurdjevic,  \newblock \emph{Geometric Control
%Theory},  \newblock Cambridge Univ. Press (1997).

%\bibitem {Jurd1} \newblock V. Jurdjevic and G. Sallet
%\newblock \emph{Controllability properties of affine systems},  
%\newblock SIAM J. Control Opt. 22 No 3 (1984), 501 - 508

\bibitem {Shattler} 
\newblock U. Ledzewick, H. Shattler, 
\newblock \emph{Optimal controls for a two compartment model for cancer	chemotherapy with quadratic objective.}, 
\newblock Proceedings of MTNS (2006), Kyoto,Japan.

\bibitem {Leitmann} \newblock G. Leitmann,  \newblock \emph{Optimization
Techniques with Application to Aerospace Systems},  \newblock Academic Press
Inc., London, 1962.

\bibitem {Markus} \newblock L. Markus,  \newblock \emph{Controllability of
multi-trajectories on Lie groups},  \newblock  Proceedings of Dynamical
Systems and Turbulence, Warwick 1980, Lecture Notes in Mathematics 898, 250-265.

%\bibitem {Oni} \newblock A. L. Onishchik and E. B. Vinberg,
%\newblock \emph{Lie groups and Lie algebras III - Structure of Lie groups and
%Lie algebras},  \newblock Springer Verlag, Berlin, 1994.

\bibitem {P-B-G-M} \newblock L. S. Pontryagin, V. G. Boltyanskii, R. V.
Gamkrelidze and E.F. Mishchenko, \newblock \emph{The mathematical theory of
optimal processes}, \newblock Interscience Publishers John Wiley \& Sons,
Inc., New Yor , A. \emph{Control and Systems Engineering A Report
on Four Decades of Contributions, }Studies in Systems, Decision and Control, 2015.

%\bibitem {SM1} \newblock L. A. B. San Martin, \newblock \emph{Algebras de Lie},
%\newblock Second Edition, Editora Unicamp, (2010).

\bibitem {Shell}K. Shell, \newbox \emph{Applications of Pontryagin's Maximum
Principle to Economics}, \newblock Mathematical Systems Theory and Economics I
and II, Volume 11/12 of the series Lecture Notes in Operations Research and
Mathematical Economics (1968), 241-292.



%\bibitem {Jurd2} \newblock H.J. Sussmann and V. Jurdjevic.
%\newblock \emph{Controllability of Nonlinear Systems}, \newblock J. Diff. Eqs.
%12 (1972), 95-116.

%	\bibitem{Suss1}
%	\newblock H.J. Sussmann.
%	\newblock \emph{A sufficient condition for local controllability},
%	\newblock SIAM J. Control and Optimization 16 No 5 (1978), 790-802.


%	\bibitem{Suss2}
%	\newblock H.J. Sussmann.
%	\newblock \emph{Lie Brackets and local controllability: A sufficient condition for scalar input systems},
%	\newblock SIAM J. Control and Optimization 21 No 5 (1983), 686-713.


%	\bibitem{Suss3}
%	\newblock H.J. Sussmann.
%	\newblock \emph{A General Theorem on Local Controllability},
%	\newblock SIAM J. Control and Optimization 25 No 1 (1985), 158-194.


%	\bibitem{Wonham}
%	\newblock W. Wonham,
%	\newblock \emph{Linear Multivariable Control: a Geometric Approach},
%	\newblock Applications of Mathematics 10 (1979).


%	\bibitem{Wunster}
%	\newblock M. W\"{u}stner,
%	\newblock \emph{On the surjectivity of the exponential function on solvable Lie groups},
%	\newblock Math. Nachr. 192 (1998), 255-266.


%\bibitem{ask}
%\newblock A .Da Silva and C. Kawan,
%\newblock \emph{Hyperbolic chain control sets on flag manifolds},
%\newblock Journal of Dynamics and Control Systems, 21 No. 4 (2015), 1-21.


%\bibitem{ask1}
%\newblock A .Da Silva and C. Kawan,
%\newblock \emph{Invariance entropy of hyperbolic control sets},
%\newblock Discrete and Continuous Dynamical Systems, 36 No.1 (2016), 97-136.


%\bibitem{Jur}
%\newblock V. Jurjevic,
%\newblock \emph{Geometric control theory},
%\newblock Cambridge University Press, (1997).


%\bibitem{Jurd}
%\newblock V. Jurjevic and H. Sussmann,
%\newblock \emph{Control systems on Lie groups},
%\newblock Journal of Differential Equations, 12 (1972), 313-329.


%\bibitem{Ka2}
%\newblock C. Kawan,
%\newblock \emph{Invariance Entropy for Deterministic Control Systems -- An Introduction},
%\newblock Lecture Notes in Mathematics, 2089, Springer-Verlag, Berlin, (2013).


%\bibitem{psm}
%\newblock M .Patrao and L.A.B. San Martin,
%\newblock \emph{Semiflows on Topological Spaces: Chain Transitivity and Semigroups},
%\newblock Journal of Dynamics and Differential Equations, 19 No.1 (2007), 155-180.


%\bibitem{sm}
%\newblock  L.A.B. San Martin,
%\newblock \emph{Order and domains of attractions of control sets in flag manifolds},
%\newblock Journal of Lie theory 8 No. 2 (1998), 335-350.


%\bibitem{smls}
%\newblock  L.A.B. San Martin and L. Seco,
%\newblock \emph{Morse and Lyapunov spectra and dynamics on flag bundles},
%\newblock Ergod. Th. \& Dynam. Sys. 30 No.3 (2010), 893-922.


%\bibitem{smt}
%\newblock  L.A.B. San Martin and P. A. Tonelli,
%\newblock \emph{Semigroup actions on Homogeneous Spaces},
%\newblock Semigroup Forum, 50 (1995), 59-88.

\end{thebibliography}
\end{document}